\documentclass[a4paper,11pt]{amsart}
\pdfoutput=1
\usepackage[utf8]{inputenc}

\usepackage{hyperref}

\usepackage[sc]{mathpazo}
\usepackage[tracking]{microtype}

\usepackage[T1]{fontenc}
\usepackage{textcomp}

\usepackage[arrow,matrix]{xy}
\usepackage{enumerate}

\theoremstyle{plain}
\newtheorem{thm}{Theorem}
\newtheorem{prop}{Proposition}[section]
\newtheorem{lem}[prop]{Lemma}

\theoremstyle{definition}
\newtheorem{defn}[prop]{Definition}

\theoremstyle{remark}
\newtheorem{rem}[prop]{Remark}

\DeclareMathOperator{\EExt}{\mathcal{E}\mathit{xt}}
\DeclareMathOperator{\Ext}{Ext}
\DeclareMathOperator{\Hom}{Hom}
\DeclareMathOperator{\id}{id}
\DeclareMathOperator{\pr}{pr}

\begin{document}
\keywords{Kodaira conjecture, conic bundles}
\subjclass[2010]{32J27, 32J17, 32G05, 32Q15}
\title[Algebraic approximation]{Algebraic approximation of Kähler threefolds}
\author{Florian Schrack}
\address{Mathematisches Institut\\ Universität Bayreuth\\ 95440 Bayreuth\\ Germany}
\email{florian.schrack@uni-bayreuth.de}
\begin{abstract}
In the present work, we investigate existence of deformations and algebraic approximability for certain uniruled Kähler threefolds. In the first part, we establish existence of infinitesimal deformations for all conic bundles with relative Picard number one over a non-algebraic compact Kähler surface~$S$ and existence of positive-dimensional families of deformations in all but some special cases. In the second part, we study the question of algebraic approximability for projective bundles over~$S$ and threefolds bimeromorphic to $\mathbb{P}_1\times S$.
\end{abstract}
\bibliographystyle{amsalpha}
\maketitle
\tableofcontents
\section{Introduction}
An interesting class of complex manifolds to study is the class of compact manifolds which arise as limits of projective manifolds in the following sense:
\begin{defn}
Let $X$ be a compact complex manifold. A proper holomorphic submersion $\pi\colon\mathcal{X}\to T$ between complex manifolds $\mathcal{X}$ and $T\ni0$ is called an \emph{algebraic approximation of $X$} if
\begin{enumerate}[(i)]
\item the central fiber $\mathcal{X}_0:=\pi^{-1}(0)$ is isomorphic to $X$ and
\item there exists a sequence $(t_k)_{k\in\mathbb{N}}\subset T$ converging to $0$ such that for each~$k$, the fiber $\mathcal{X}_{t_k}:=\pi^{-1}(t_k)$ is a projective manifold. \end{enumerate}
We call $X$ \emph{algebraically approximable} if there exists an algebraic approximation of $X$. 
\end{defn}

It is a corollary of Kodaira's classification of complex surfaces that every compact Kähler surface is algebraically approximable. A new, more conceptual proof of this result not relying on classification results has recently been given by Buchdahl in \cite{Buc06} and \cite{Buc08}.

It was conjectured for a long time that indeed every compact Kähler manifold is algebraically approximable. This was shown to be false by C.~Voisin in~\cite{Voi04}, where she proved that in every dimension $\ge 4$ there exist compact Kähler manifolds not even having the homotopy type of a projective manifold.

It is still open whether every compact Kähler threefold is algebraically approximable. Since the classification theory of Kähler threefolds is still a little bit vague at the time being, the most one can possibly hope for is settling a reasonably large class of threefolds carrying some special structure.

In the present article, we undertake a first attempt towards this goal. Our strategy is to use Kodaira's approximation results for surfaces in order to construct algebraic approximations for threefolds carrying some type of $\mathbb{P}_1$-fibration over a Kähler surface.

An intermediate problem is to study the existence of deformations of these threefolds. For conic bundles, we establish the existence of infinitesimal deformations by a careful study of the geometry of discriminant loci in section~\ref{sec:conic-bundl-discr}:
\begin{thm}\label{thm:infpos}
Let $S$ be a compact non-algebraic Kähler surface and $f\colon X\to S$ a conic bundle with $\rho(X)=\rho(S)+1$. Then $H^1(T_X)\ne0$.
\end{thm}
In most cases, we can even show the existence of a positive-dimensional family of deformations:
\begin{thm}\label{thm:infconi}
Let $S$ be a compact non-algebraic Kähler surface and $f\colon X\to S$ a conic bundle with $\rho(X)=\rho(S)+1$. Then
\[h^1(T_X)>h^2(T_X),\]
except possibly in the following cases:
\begin{enumerate}[(i)]
\item $S$ is a torus and $E:=f_*(K_{X/S}^*)$ is projectively flat, i.e.~$\mathbb{P}(E)$ is given by a representation $\pi_1(S)\to\mathrm{PGL}(3,\mathbb{C})$. Furthermore, $X$ is not isomorphic to $\mathbb{P}(V)$ for any rank-two bundle $V$ on $S$.
\item $S$ is a minimal properly elliptic surface (i.e.~$\kappa(S)=1$) and all singular fibers of the elliptic fibration are multiples of smooth elliptic curves.
\end{enumerate}
\end{thm}

By studying vector bundles over K3 surfaces and tori in section~\ref{sec:vector-bundles-k3}, we lay the basis for understanding projectivized rank-two bundles over surfaces of Kodaira dimension $0$. We obtain the following result:
\begin{thm}\label{thm:projr2b}
Let $S$ be a compact Kähler surface with $\kappa(S)=0$ and $V$ a holomorphic rank-two vector bundle on $S$. Then the projective bundle $\mathbb{P}(V)$ is algebraically approximable.
\end{thm}

By combining techniques from the previous sections, we extend our results in various directions in section~\ref{sec:algebr-appr-cert}. First, we prove algebraic approximability of threefolds bimeromorphic to a product of $\mathbb{P}_1$ with a compact Kähler surface:
\begin{thm}\label{thm:main}
Let $X$ be a compact Kähler threefold for which there exists a compact Kähler surface $S$ and a bimeromorphic map $\phi\colon \mathbb{P}_1\times S\dashrightarrow X$. Then $X$ is algebraically approximable.
\end{thm}
Finally, we prove algebraic approximability for certain types of conic bundles:
\begin{thm}\label{thm:ellcon}
Let $S$ be a compact Kähler surface with an elliptic fibration $r\colon S\to C$ over a smooth curve $C$. If $f\colon X\to S$ is a conic bundle such that the rank-three bundle $f_*(K_{X/S}^*)$ is trivial when restricted to the general fiber of $r$, then $X$ is algebraically approximable.
\end{thm}
\begin{thm}\label{thm:splitcon}
Let $S$ be a K3 surface and $f\colon X\to S$ a conic bundle such that $f_*(K_{X/S}^*)$ splits as a direct sum of line bundles. Then $X$ is algebraically approximable.
\end{thm}

\noindent{\bf Acknowledgements.}
In large parts, this paper is a condensed version of my Ph.D.~thesis prepared at the University of Bayreuth~\cite{Schrack10}. I would like to thank my advisor Thomas Peternell for many fruitful discussions and suggestions.

The work was supported by DFG Forschergruppe 790 ``Classification of algebraic surfaces and compact complex manifolds.''

\section{The concept of algebraic approximability}
In this section, we first review some approximability results for elliptic surfaces we will use in the following. After that we investigate the behavior of algebraic approximability with respect to certain types of maps between complex manifolds.

We start with general remark concerning the algebraic approximation of \emph{Kähler} manifolds:
\begin{rem}
Let $\mathcal{S}\to T\ni 0$ be a deformation of a compact Kähler manifold~$\mathcal{S}_0$. Then for any $t\in T$ in a sufficiently small neighborghood of~$0$, the fiber $\mathcal{S}_t$ is also Kähler. Since a compact Kähler manifold is projective if and only if it is Moishezon, finding an approximation of~$\mathcal{S}_0$ by projective manifolds is equivalent to finding an approximation by Moishezon manifolds.
\end{rem}

\subsection{The elliptic surface case}
\begin{prop}\label{prop:elsurf}
Let $S$ be a compact Kähler surface with an elliptic fibration $r\colon S\to C$ over a smooth curve $C$. Then there exist a complex manifold $T\ni0$ and a deformation $\pi\colon\mathcal{S}\to T$ of $S=\mathcal{S}_0$ with the following properties:
\begin{enumerate}[(i)]
\item There is an elliptic fibration $R\colon\mathcal{S}\to C\times T$ with $\pr_2\circ R=\pi$ and $R|_{\mathcal{S}_0}=r$.
\item For every point $c\in C$ there exists an open neighborhood $U\subset C$ of $c$ such that $R^{-1}(U\times T)\cong r^{-1}(U)\times T$ over $C\times T$.
\item The set of points $t\in T$ where $\mathcal{S}_t$ is projective is dense in $T$.
\end{enumerate}
\end{prop}
\begin{proof}
This proposition is the essence of an intricate study of elliptic fibrations first carried out by Kodaira (cf.~\cite{FM94} for details). A much more elegant proof has been given by Buchdahl in~\cite{Buc08}.
\end{proof}

\subsection{Some general criteria}
The following criterion follows easily from Kodaira's result about stability of subspaces:
\begin{prop}
Let $X\subset Y$ be an inclusion of compact complex manifolds. If $Y$ is algebraically approximable and $H^1(N_{X|Y})=0$, then also $X$ is algebraically approximable.
\end{prop}
As étale covers are topological objects, an algebraic approximation can obviously be lifted to finite étale covers:
\begin{prop}
Let $X$ be a compact complex manifold which is algebraically approximable. Then any finite étale cover of $X$ is also algebraically approximable.
\end{prop}
The following is slightly more complicated:
\begin{prop}
Let $f\colon X\to Y$ be a holomorphic surjection from a compact Kähler manifold $X$ to some manifold $Y$. Assume that $f$ has connected fibers and that for general $y\in Y$, $X_y$ is projective with $H^1(\mathcal{O}_{X_y})=0$. Suppose furthermore that $\Ext^2(\Omega^1_{X/Y},\mathcal{O}_X)=0$. Then if $Y$ is algebraically approximable, $X$ is also algebraically approximable.
\end{prop}
\begin{proof}
Let $\psi\colon\mathcal{Y}\to T\ni0$ be an algebraic approximation of $Y=\mathcal{Y}_0$. By \cite[Theorem~6.1]{Hor74}, the assumption $\Ext^2(\Omega^1_{X/Y},\mathcal{O}_X)=0$ implies (after shrinking $T$ sufficiently) that there exist a deformation $\pi\colon\mathcal{X}\to T\ni 0$ of $X=\mathcal{X}_0$ and a holomorphic map $F\colon\mathcal{X}\to\mathcal{Y}$ over $T$ such that $F|_{\mathcal{X}_0}=f$. It is easy to see that $F$ has connected fibers and $H^1(\mathcal{O}_{\mathcal{X}_y})=0$ for general $y\in\mathcal{Y}$. Since $X$ is Kähler, also $\mathcal{X}_t$ is Kähler for sufficiently small~$t$. Now if $t\in T$ is such that $\mathcal{Y}_t$ is projective, $\mathcal{X}_t$ is also projective by~\cite[Proposition~7]{Fuj83a}.
\end{proof}
We also state the following property:
\begin{prop}
Let $X$ be a compact Kähler manifold, $\hat{X}\to X$ the blow-up of $X$ along a compact submanifold $Y\subset X$ of codimension $\ge2$. Then $\hat{X}$ is algebraically approximable if and only if there exists a $Y$-stable algebraic approximation of $X$.
\end{prop}
One can generalize the concept of algebraic approximability to arbitrary complex spaces. By using generalizations of Horikawa's deformation theory from~\cite{Ran91}, one can for instance prove the following result:
\begin{prop}
Let $X$ be a three-dimensional normal complex variety having at most terminal singularities and $f\colon X\dashrightarrow X'$ a composition of Mori contractions and flips. Then, if $X$ is algebraically approximable, so is $X'$.
\end{prop}

\section{Conic bundles: discriminant loci and deformations}
\label{sec:conic-bundl-discr}
In this section, we want to investigate the geometry of discriminant loci of conic bundles over non-algebraic surfaces: Let $f\colon X\to S$ be a conic bundle over a compact surface $S$ (i.e.~$X$ is a complex manifold and every fiber of $f$ is isomorphic to a curve of degree $2$ in $\mathbb{P}_2$). If we let $E:=f_*{K_{X/S}^*}$, then $X$ is a divisor on $\mathbb{P}(E)$ inside the linear system $\lvert\mathcal{O}_{\mathbb{P}(E)}(2)\otimes\pi^*\det E^*\rvert$, where $\pi\colon\mathbb{P}(E)\to S$ is the natural projection. Thus $X$ is given by a section $\sigma\in H^0(S^2E\otimes\det E^*)$. Via the canonical embedding
\[S^2E\otimes\det E^*\subset E\otimes E\otimes\det E^*\cong\mathcal{H}\mathit{om}_{\mathcal{O}_S}(E^*,E\otimes\det E^*),\]
the section $\sigma$ induces a non-zero section
\[\det\sigma\in\mathrm{Hom}(\det E^*,\det (E\otimes\det E^*))\cong H^0(\det E^*),\]
which defines a divisor $\Delta_f$ on $S$, the so-called \emph{discriminant locus} of $f$. Obviously, $\Delta_f$ consists of those points $s\in S$ for which the fiber $X_s$ is a singular conic.

The discriminant locus has an interesting property:
\begin{prop}\label{prop:pdisc}
For any conic bundle $f\colon X\to S$ over a compact surface $S$ with $\rho(X)=\rho(S)+1$, the discriminant locus $\Delta_f$ is a reduced divisor with the property that any smooth rational component of $\Delta_f$ meets other components of $\Delta_f$ in at least two points.
\end{prop}
\begin{proof}
cf.~{\cite[Remark~4.2 and Lemma~4.5]{Miy83}}
\end{proof}
This motivates the following definition:
\begin{defn}
A divisor $D$ on a compact complex surface is said to have \emph{property~(P)} if it is reduced and for any smooth rational component $C$ of $D$ we have
\[C\cdot(D-C)\ge2.\]
\end{defn}
Using the terminology of this definition, Proposition~\ref{prop:pdisc} states that $\Delta_f$ has property~(P).

Property~(P) is preserved under blow-downs:
\begin{prop}\label{prop:pblow}
Let $S$ be a compact surface and $p\colon S\to S'$ the blow-down of a $(-1)$-curve $C_0$ on $S$. Then for any divisor $D$ on $S$ which has property~(P), the divisor $D':=p_*D$ on $S'$ also has property~(P).
\end{prop}
\begin{proof}
Since $D$ is reduced, there are distinct irreducible curves $C_1$, $\dots$, $C_k$ different from $C_0$ such that
\[\text{either}\quad D=\sum_{i=1}^kC_i\qquad\text{or}\quad D=C_0+\sum_{i=1}^kC_i.\]
If we let $C_1':=p_*C_1$, $\dots$, $C_k':=p_*C_k$, then $C_1'$, $\dots$, $C_k'$ are distinct irreducible curves on $S'$ and we have
\[D'=p_*D=\sum_{i=1}^kC_i'.\]
If now $C_i'$ is smooth rational for some $i$, then also $C_i$ is smooth rational, so by hypothesis we have
\[C_i\cdot(D-C_i)\ge2.\]
From this we obtain:
\[\begin{split}
C_i'\cdot(D'-C_i')&=p^*C_i'\cdot p^*(p_*D-C_i')\\
&=p^*C_i'\cdot(D-C_i)\\
&=C_i\cdot(D-C_i)+(C_0\cdot C_i)C_0\cdot(D-C_i).
\end{split}\]
The claim now follows easily because $C_0\cdot C_i\in\{0,1\}$ and either $C_0$ is not a component of $D$, whence $C_0\cdot(D-C_i)\ge0$, or $C_0$ is a smooth rational component of $D$ and thus
\[C_0\cdot(D-C_i)=C_0\cdot(D-C_0)+C_0^2-C_0\cdot C_i\ge0\]
by hypothesis.
\end{proof}
We shall use Propositions \ref{prop:pdisc} and \ref{prop:pblow} to study discriminant loci of conic bundles over surfaces by induction over blow-ups. For surfaces with algebraic dimension~$0$, we obtain:
\begin{lem}\label{lem:disca0}
Let $S$ be a compact Kähler surface with algebraic dimension $a(S)=0$. Then any conic bundle $f\colon X\to S$ with $\rho(X)=\rho(S)+1$ is a $\mathbb{P}_1$-bundle.
\end{lem}
\begin{proof}
By Proposition~\ref{prop:pdisc} the discriminant locus $\Delta_f$ has property~(P), so it is sufficient to show that if $D$ is a divisor on $S$ having property~(P), then necessarily $D=0$.

Since $a(S)=0$, however, any connected curve on $S$ is a tree of rational curves which never satisfies property~(P).
\end{proof}
Using more intricate arguments, we can also handle the case where $S$ has algebraic dimension $1$:
\begin{lem}\label{lem:disca1}
Let $S$ be a compact complex surface with algebraic dimension $a(S)=1$ and $f\colon X\to S$ a conic bundle with $\rho(X)=\rho(S)+1$. Then we have
\[\Delta_f^2\ge3\Delta_f\cdot K_S+4K_S^2.\]
\end{lem}
\begin{proof}
By Proposition~\ref{prop:pdisc}, it is sufficient to show that for any divisor $D$ on $S$ satisfying property~(P), we have
\begin{equation}\label{eq:discineq}
D\cdot(D-3K_S)-4K_S^2\ge0.
\end{equation}

In order to show inequality~\eqref{eq:discineq}, we first assume $S$ to be minimal. Since $a(S)=1$, the surface $S$ is elliptic. Let
$r\colon S\to B$
be the (unique) elliptic fibration over a smooth curve~$B$. Then any irreducible curve on $S$ is contained in a fiber of $r$. So, by Kodaira's classification of elliptic fibers, any irreducible curve on $S$ is either a $(-2)$-curve or numerically proportional to a fiber of $r$. If $F$ is a fiber of $r$ and $L$ is an arbitrary line bundle on $S$, then $F\cdot L=0$, so for numerical purposes we can assume that $D$ is a sum of distinct $(-2)$-curves. By an easy calculation, property~(P) then implies that $D^2=0$. This in turn implies that $D\cdot L=0$ for any line bundle $L$ on $S$, so inequality~\eqref{eq:discineq} follows from $K_S^2\le0$.

If $S$ is not minimal, we let $p\colon S\to S'$ be the blow-down of a $(-1)$-curve $C_0\subset S$. By Proposition~\ref{prop:pblow}, the divisor $D':=p_*D$ has property~(P), so we may inductively assume that
\begin{equation}\label{eq:indineq}
D'\cdot(D'-3K_{S'})-4K_{S'}^2\ge0.
\end{equation}
Now  we have $K_S=p^*K_{S'}+C_0$. Furthermore, if we write $D=p_*^{-1}D'+\varepsilon C_0$ (note that $\varepsilon\in\{0,1\}$ because D is reduced) and let $\mu$ be the multiplicity of the divisor $D'$ at the point $p(C_0)$, we obtain
\begin{equation}\label{eq:pulb}
D=p^*D'+(\varepsilon-\mu)C_0.
\end{equation}
So we can calculate
\[D\cdot(D-3K_S)-4K_S^2=D'\cdot(D'-3K_S)-4K_S^2+4+(\mu-\varepsilon)(\varepsilon-\mu-3).\]
By inequality~\eqref{eq:indineq}, in order to show inequality~\eqref{eq:discineq}, it thus remains to show that
\begin{equation}\label{eq:mueps}
(\mu-\varepsilon)(\mu-\varepsilon+3)\le 4.
\end{equation}
Since $a(S')=a(S)=1$, the surface~$S'$ is also elliptic with all irreducible curves contained in fibers of the elliptic fibration. By Kodaira's classification of fibers, every reduced curve on $S'$ has multiplicity at most $3$ in any of its points, so in particular we have $\mu\le 3$.

Because $C_0$ is smooth rational and $C_0\cdot(D-C_0)=\mu-\varepsilon+1$, property~(P) yields the implication
\begin{equation}\label{eq:pimp}
\varepsilon=1\Rightarrow\mu\ge2.
\end{equation}

Hence, in order to prove inequality~\eqref{eq:mueps}, it remains to exclude the following two cases:
\begin{enumerate}[(i)]
\item $\mu=3$,
\item $\mu=2$ and $\varepsilon=0$.
\end{enumerate}

We first assume $\mu=3$. Using Kodaira's classification of singular elliptic fibers, we find that the only possible way for the reduced divisor $D'$ to contain a point of multiplicity $3$ is that $D'$ contains three smooth rational components $C_1'$, $C_2'$ and $C_3'$ intersecting in an ordinary triple point. Now, using implication~\eqref{eq:pimp} for further blow-downs of $S'$, we see that $D'$ cannot contain any $(-1)$-curve lying in the same elliptic fiber as the $C_i'$ but distinct from the $C_i'$. Thus we find that the divisor $C_1'+C_2'+C_3'$ is a connected component of $D'$. If we let $C_i:=p^{-1}_*C_i'$, then $C_i$ is a smooth rational component of $D$ such that, by equation~\eqref{eq:pulb},
\[C_i\cdot(D-C_i)=(p^*C_i'-C_0)\cdot(p^*(D'-C_i')+(\varepsilon-2)C_0)=\varepsilon,\]
contradicting property~(P).

It remains to consider the case $\mu=2$. We must show that in this case, $\varepsilon=1$, i.e.~$C_0$ is a component of $D$. To prove this, we can assume that $D'$ consists of exactly one connected component, namely the one containing the point $p(C_0)$. By Kodaira's classification and implication~\eqref{eq:pimp}, we thus obtain the following possibilities for $D'$:
\begin{enumerate}[(i)]
\item $D'$ is an irreducible curve having an ordinary double point or a cusp at $p(C_0)$. In this case the strict transform $p^{-1}_*D'$ is a smooth rational curve, so property~(P) implies that $C_0$ must be a component of $D$.
\item $D'=C_1'+C_2'$ where $C_1'$ and $C_2'$ are smooth rational curves intersecting only at $p(C_0)$ (with multiplicity $2$). Then $p^{-1}_*D'$ consists of two smooth rational curves intersecting transversally at exactly one point, so again by property~(P), $C_0$ must be a component of $D$.
\item $D'$ is a cycle of smooth rational curves. Then $p^{-1}_*D'$ is a chain of rational curves. In this case, we can apply property~(P) to the curves at the end of this chain to conclude that $C_0$ must be a component of~$D$.\qedhere
\end{enumerate}
\end{proof}

We can use the preceding lemmas to estimate the dimension of the deformation space of a conic bundle via the Riemann-Roch theorem. To do so, we need the following basic result:
\begin{prop}\label{prop:conbun}
Let $f\colon X\to S$ be a conic bundle over a compact surface $S$. Then we have $\Ext^3(\Omega^1_{X/S},\mathcal{O}_X)=0$ and $H^3(T_X)=0$.
\end{prop}
\begin{proof}
We first prove that the relative cotangent sheaf $\Omega^1_{X/S}$ is torsion-free:  If we let $E:=f_*K_{X/S}^*$, then $X$ is a submanifold of $\mathbb{P}(E)$ and we get a short exact sequence
\[0\longrightarrow N^*_{X|\mathbb{P}(E)}\longrightarrow\Omega^1_{\mathbb{P}(E)/S}|_X\longrightarrow\Omega^1_{X/S}\longrightarrow0.\]
Dividing out the torsion of $\Omega^1_{X/S}$ yields the sequence
\[0\longrightarrow K\longrightarrow\Omega^1_{\mathbb{P}(E)/S}|_X\longrightarrow\Omega^1_{X/S}/\mathrm{tor}\longrightarrow0,\]
where $K$ is a reflexive rank-$1$ subsheaf (and thus a sub-line bundle) of the rank-$2$ bundle $\Omega^1_{\mathbb{P}(E)/S}|_X$. We obviously have
\begin{equation}\label{eq:torincl}
N^*_{X|\mathbb{P}(E)}\subset K.
\end{equation}
Smoothness of $X$ readily implies that $\Omega^1_{X/S}$ is locally free (and thus torsion-free) outside codimension~$2$. This means that \eqref{eq:torincl} is an isomorphism outside codimension~$2$. Since the left and right sides of~\eqref{eq:torincl} are both line bundles, this implies that \eqref{eq:torincl} is actually an equality, so $\Omega^1_{X/S}$ is torsion-free.

Now, for a general point $s\in S$, the fiber~$X_s$ is isomorphic to $\mathbb{P}_1$, and thus $K_X|_{X_s}\cong\Omega^1_{X/S}|_{X_s}\cong\Omega^1_{X_s}\cong\mathcal{O}_{\mathbb{P}_1}(-2)$. So we have $H^0((\Omega^1_{X/S}\otimes K_X)|_{X_s})=0$ for general $s$ and thus $H^0(\Omega^1_{X/S}\otimes K_X)=0$ by torsion-freeness of $\Omega^1_{X/S}$.

Serre duality yields $\Ext^3(\Omega^1_{X/S},\mathcal{O}_X)\cong H^0(\Omega^1_{X/S}\otimes K_X)^*$, and the first part of the claim follows.

Finally, again by Serre duality, $H^3(T_X)=0$ is equivalent to $H^0(\Omega^1_X\otimes K_X)=0$, which is clearly true since the bundle $\Omega^1_X\otimes K_X$ is negative on any smooth fiber of~$f$.
\end{proof}
We have now gathered all the necessary tools to estimate the dimension of the deformation space of a non-algebraic conic bundle:
\begin{prop}\label{prop:exdefo}
Let $S$ be a compact non-algebraic Kähler surface and $f\colon X\to S$ a conic bundle with $\rho(X)=\rho(S)+1$. Then
\[h^1(T_X)\ge h^2(T_X).\]
If $h^1(T_X)=h^2(T_X)$, then $H^0(T_X)=0$ and
\[c_1^2(S)=c_2(S)=c_1^2(f_*(K_{X/S}^*))=c_2(f_*(K_{X/S}^*))=0.\]
\end{prop}
\begin{proof}
The pushforward
\[E:=f_*(K_{X/S}^*)\]
is a rank-$3$ vector bundle on $S$ such that there exists a canonical embedding $X\subset\mathbb{P}(E)$. We have $H^3(T_X)=0$ by Proposition~\ref{prop:conbun}, so by a standard Riemann-Roch calculation (see~\cite[Proposition~4.2]{Schrack10} for details), we obtain the formula
\begin{equation}\label{eq:riero}
\begin{split}
h^1(T_X)-h^2(T_X)&=h^0(T_X)+\bigl(c_2(E)-c_1(E)c_1(S)-\tfrac{4}{3}c_1^2(S)\bigr)\\
&\quad-\tfrac{2}{3}c_1^2(S)+7\chi(\mathcal{O}_S).
\end{split}
\end{equation}
$S$ being a non-algebraic Kähler surface, we have $c_1^2(S)\le0$ and $\chi(\mathcal{O}_S)\ge0$. Furthermore, by \cite[Théorème 3.1]{BL87}, we have
\begin{equation}\label{eq:banlep}
c_2(E)-\tfrac{1}{3}c_1^2(E)\ge0,
\end{equation}
so, by~\eqref{eq:riero}, in order to prove $h^1(T_X)\ge h^2(T_X)$, it remains to show that
\begin{equation}\label{eq:chern}
c_1^2(E)\ge 3c_1(E) c_1(S)+4c_1^2(S).
\end{equation}
If $a(S)=1$, this follows directly from Lemma~\ref{lem:disca1} by observing that $\Delta_f\in\lvert\det E^*\rvert$;
if $a(S)=0$, then $c_1(E)=0$ by Lemma~\ref{lem:disca0} and the inequality follows because $c_1^2(S)\le0$.

Now if $h^1(T_X)=h^2(T_X)$, all summands on the right side of~\eqref{eq:riero} must be zero, so we directly obtain $h^0(T_X)=0$ and $c_1^2(S)=c_2(S)=0$. Since $S$ is non-algebraic, $c_1^2(S)=0$ implies that $c_1(S)\cdot L=0$ for any line bundle $L$ on~$S$. Furthermore, we must have equality in \eqref{eq:banlep} and \eqref{eq:chern}, so we also obtain $c_1^2(E)=c_2(E)=0$.
\end{proof}
We can now prove two of the theorems announced in the introduction:
\begin{proof}[Proof of Theorem~\ref{thm:infpos}]
By Proposition~\ref{prop:exdefo}, we only need to exclude the case $h^1(T_X)=h^2(T_X)=0$: Suppose $H^1(T_X)=H^2(T_X)=0$. Since we have $\Ext^3(\Omega^1_{X/S},\mathcal{O}_X)=0$
by Proposition~\ref{prop:conbun}, applying $\Hom(\cdot,\mathcal{O}_X)$ to the relative cotangent sequence
\[0\longrightarrow f^*\Omega^1_S \longrightarrow \Omega^1_X \longrightarrow \Omega^1_{X/S} \longrightarrow 0\]
yields
\[0=H^2(T_X)\longrightarrow H^2(f^* T_S) \longrightarrow \Ext^3(\Omega^1_{X/S},\mathcal{O}_X)=0\]
and thus (using Serre duality)
\[H^0(\Omega^1_S\otimes K_S)^*\cong H^2(T_S) \cong H^2(f^*T_S) =0.\]
We have $c_1^2(S)=0$ by Proposition~\ref{prop:exdefo}, so $S$ is minimal. Together with $c_2(S)=0$ and $H^2(T_S)=0$, this implies that $S$ is properly elliptic by surface classification theory. Let $r\colon S\to C$ be the elliptic fibration of $S$ over a smooth curve $C$. Tensoring the relative cotangent sequence of the elliptic fibration $r$ by $K_S$ and taking global sections, we obtain the long exact sequence 
\[0\longrightarrow H^0(r^*K_C\otimes K_S) \longrightarrow H^0 (\Omega^1_S\otimes K_S) \longrightarrow H^0(\Omega^1_{S/C}\otimes K_S) \longrightarrow\dots,\]
so we obtain
\begin{equation}\label{eq:canvan}
H^0(r^*K_C\otimes K_S)=0.
\end{equation}
Now, by relative duality,
\begin{equation}\label{eq:reldual}
r_*K_S=K_C\otimes (R^1r_*\mathcal{O}_S)^*,
\end{equation}
where by Riemann-Roch on~$C$,
\[
\begin{split}
\deg(R^1r_*\mathcal{O}_S)^*&=-\deg(R^1r_*\mathcal{O}_S)=-\chi(R^1r_*\mathcal{O}_S)+\chi(\mathcal{O}_C)\\ &=h^0(\mathcal{O}_C)-h^1(\mathcal{O}_C)-h^0(R^1r_*\mathcal{O}_S)+h^1(R^1r_*\mathcal{O}_S).
\end{split}
\]
Now the Leray spectral sequence for~$r$ yields
\[\begin{split}
h^0(\mathcal{O}_S)&=h^0(\mathcal{O}_C),\\ h^1(\mathcal{O}_S)&=h^1(\mathcal{O}_C)+h^0(R^1r_*\mathcal{O}_S),\\
h^2(\mathcal{O}_S)&=h^1(R^1r_*\mathcal{O}_S),
\end{split}\]
hence $\deg(R^1r_*\mathcal{O}_S)^*=\chi(\mathcal{O}_S)=0$ (cf.~\cite[p.~213f.]{BPV04}). We obtain
\[\deg(r_*(r^*K_C\otimes K_S))=4g(C)-4,\]
thus by equation~\eqref{eq:canvan}, we conclude that $g(C)\le1$. On the other hand, as $S$ is a non-algebraic Kähler surface, we must have $H^0(K_S)=H^{2,0}(S)\ne0$, so \eqref{eq:reldual} implies $g(C)\ge1$. So $C$ must be an elliptic curve, in particular we have $K_C\cong\mathcal{O}_C$.
But then, equation~\eqref{eq:canvan} contradicts $H^0(K_S)\ne0$.
\end{proof}
In order to prove Theorem~\ref{thm:infconi}, we need to describe the exceptional cases of Proposition~\ref{prop:exdefo} in more detail:
\begin{proof}[Proof of Theorem~\ref{thm:infconi}]
Suppose that $h^1(T_X)=h^2(T_X)$. Then by Proposition~\ref{prop:exdefo}, we have $c_1^2(S)=0$, which yields that $S$ is a minimal surface. From $c_2(S)=0$, we obtain that $S$ must either be a torus or a minimal properly elliptic surface such that all singular fibers of the elliptic fibration on $S$ are multiples of smooth elliptic curves (cf.~\cite[p.~14]{Kod63b}).

Now suppose that $S$ is a torus. By Proposition~\ref{prop:exdefo}, we have $c_1^2(E)=c_2(E)=0$. By~\cite[Theorem~5.12]{Yan89}, this implies that $E$ is projectively flat. Now, suppose that there exists a rank-2 bundle $V$ on $S$ such that $X\cong\mathbb{P}(V)$. An easy calculation shows that we have $H^q(X, T_{X/S})\cong H^q(S, S^2V\otimes\det V^*)$ for all~$q$, so by $H^0(T_X)=0$, it follows that $H^0(S^2V\otimes\det V^*)=0$ and therefore, by Serre duality, also $H^2(X,T_{X/S})\cong H^2(S,S^2V\otimes\det V^*)=0$. The relative tangent sequence then yields $h^1(T_X)\ge h^1(T_S)=4$ and $h^2(T_X)\le h^2(T_S)=2$, so in particular $h^1(T_X)>h^2(T_X)$.
\end{proof}

\section{Vector bundles on K3 surfaces and tori}
\label{sec:vector-bundles-k3}
The aim of this section is to prove algebraic approximability of projectivized rank-two bundles over Kähler surfaces of Kodaira dimension $0$. The most important step is to examine vector bundles over K3 surfaces and tori:

Let $S$ be a K3 surface or a two-dimensional torus and $V$ a vector bundle on $S$. The objective is to study whether one can construct an algebraic approximation $\pi\colon\mathcal{S}\to T\ni0$ of $S=\mathcal{S}_0$ such that there exists a vector bundle $\mathcal{V}$ on $\mathcal{S}$ with $\mathcal{V}|_{\mathcal{S}_0}\cong V$. We will show that this is always possible (up to a twist of $V$ by a line bundle) if $V$ has rank $2$.

We begin by examining the case of line bundles. In order to deal with higher-rank bundles later on, it is important to have the following statement about ``simultaneous extendability'' of line bundles:
\begin{prop}\label{prop:k3torlb}
Let $S$ be a K3 surface or a torus. Let $\mathcal{K}\ni0$ be the Kuranishi space and $\pi\colon\mathcal{S}\to\mathcal{K}$ a versal deformation of~$S=\mathcal{S}_0$.
Then there exists an $(h^{1,1}(S)-\rho(S))$-dimensional smooth analytic subset $T$ of~$\mathcal{K}$ containing $0$ such that for any line bundle $L$ on $S$, there is a line bundle $\mathcal{L}$ on $\mathcal{S}_T:=\pi^{-1}(T)$ with $\mathcal{L}|_{\mathcal{S}_0}\cong L$.
\end{prop}
\begin{proof}
We let $h:=h^{1,1}(S)$. Then $(\mathcal{K},0)$ is an $h$-dimensional smooth germ of a complex manifold and we have a $C^\infty$-isomorphism
\[
\mathcal{S}\cong_{C^\infty}S\times \mathcal{K}.
\]
This isomorphism allows us to naturally identify $H^2(\mathcal{S}_t,\mathbb{C})$ with $H^2(S,\mathbb{C})$ for all~$t\in\mathcal{K}$.
If we now denote by $\langle\cdot,\cdot\rangle$ the $\mathbb{C}$-bilinear extension to $H^2(S,\mathbb{C})$ of the standard intersection form on $H^2(S,\mathbb{Z})$, then by Torelli's theorem, the period map
\begin{align*}
\mathcal{P}\colon\mathcal{K}&\to\mathbb{P}(H^2(S,\mathbb{C})),\\
t&\mapsto H^{2,0}(\mathcal{S}_t),
\end{align*}
gives a local isomorphism of $\mathcal{K}$ onto the period domain
\[\Omega:=\{\,[\omega]\in\mathbb{P}(H^2(S,\mathbb{C}))\mid \langle\omega,\omega\rangle=0,\langle\omega,\overline{\omega}\rangle>0\,\}.\]
We now let
\[T:=\mathcal{P}^{-1}(NS(S)^\perp\cap\Omega),\]
where $NS(S)^\perp\subset\mathbb{P}(H^2(S,\mathbb{C}))$ is the orthogonal space with respect to the bilinear form~$\langle\cdot,\cdot\rangle$. Since $\langle\cdot,\cdot\rangle$ is non-degenerate, the projective space $NS(S)^\perp$ has dimension $h+1-\rho(S)$. Since $NS(S)^\perp$ is defined over $\mathbb{R}$, the inequality $\langle\omega,\overline{\omega}\rangle>0$ in the definition of~$\Omega$ implies that $NS(S)^\perp\cap\Omega$ is smooth at every point. This implies that $T$ is smooth of dimension $h-\rho(S)$.

Now for any line bundle $L$ on $S$ and for any $t\in T$, we have that
\[\langle c_1(L),\mathcal{P}(t)\rangle=0,\]
i.e.~$c_1(L)\in H^{1,1}(\mathcal{S}_t)$ by orthogonality of the Hodge decomposition. We consider the following excerpt from the push-forward of the exponential sequence on $\mathcal{S}_T$ via the restriction $\pi_T:=\pi|_{\mathcal{S}_T}$:
\[R^1\pi_{T*}\mathcal{O}\stackrel{\psi}{\longrightarrow} R^1\pi_{T*}\mathcal{O}^*\stackrel{c_1}{\longrightarrow} R^2\pi_{T*}\mathbb{Z}\stackrel{\varphi}\longrightarrow R^2\pi_{T*}\mathcal{O}.\]
The constant extension of $c_1(L)$ defines a section $\zeta\in H^0(R^2\pi_{T*}\mathbb{Z})$ which satisfies $\varphi(\zeta)=0$ (because $c_1(L)\in H^{1,1}(\mathcal{S}_t)$ for all $t\in T$). By the above sequence, we obtain a section in $H^0(R^1\pi_{T*}\mathcal{O}^*)$ which, by the Leray spectral sequence, gives a line bundle $\tilde{\mathcal{L}}$ on $\mathcal{S}_T$ with $c_1(\tilde{\mathcal{L}}|_{\mathcal{S}_0})=c_1(L)$. Thus the line bundle $L\otimes\tilde{\mathcal{L}}^*|_{\mathcal{S}_0}$ is numerically trivial on $\mathcal{S}_0$, so is given by an element $\nu\in H^1(\mathcal{O}_{\mathcal{S}_0})$. The sheaf $R^1\pi_{T*}\mathcal{O}$ is locally free, so $\nu$ can be extended to give a section $\tilde{\nu}\in H^0(R^1\pi_{T*}\mathcal{O})$. If we let $\mathcal{N}$ be the line bundle on $\mathcal{S}$ given by $\psi(\tilde{\nu})$, then $\mathcal{L}:=\tilde{L}\otimes \mathcal{N}$ is a line bundle on $\mathcal{S}$ with $\mathcal{L}|_{\mathcal{S}_0}\cong L$.
\end{proof}
\begin{rem}
It is well-known that any non-trivial deformation of a K3 surface or a torus is an algebraic approximation. So, since $\rho(S)<h^{1,1}(S)$ for any non-algebraic surface $S$, the deformation~$\pi$ given by Proposition~\ref{prop:k3torlb} is indeed an algebraic approximation.
\end{rem}
In order to deal with rank-two bundles arising as extensions of line bundles, it is important to achieve constancy of cohomology dimensions for the extended line bundles. This is automatic for K3 surfaces:
\begin{prop}\label{prop:k3coh}
Let $S$ be a non-algebraic K3 surface, $T\ni0$ a complex manifold and $\pi\colon\mathcal{S}\to T$ a deformation of $S=\mathcal{S}_0$ with the property that for any effective divisor~$D$ on $S$ there exists a line bundle $\mathcal{L}_D$ on $\mathcal{S}$ with $\mathcal{L}_D|_{\mathcal{S}_0}\cong\mathcal{O}_S(D)$. Then for any line bundle $\mathcal{L}$ on $\mathcal{S}$, there exists an open neighborhood $U\subset T$ of $0$ such that for every~$q$, the map $U\ni t\mapsto h^q(\mathcal{L}|_{\mathcal{S}_t})$ is constant.
\end{prop}
\begin{proof}
For each $t\in T$, we let $\mathcal{L}_t:=\mathcal{L}|_{\mathcal{S}_t}$. If both $h^0(\mathcal{L}_0)\ne0$ and $h^2(\mathcal{L}_0)\ne0$, then $\mathcal{L}_0\cong\mathcal{O}_{\mathcal{S}_0}$. But then all $\mathcal{L}_t$ must be numerically trivial and hence trivial because $h^1(\mathcal{O}_{\mathcal{S}_t})=0$ for all $t$.

So we can assume that either $h^0(\mathcal{L}_0)=0$ or $h^2(\mathcal{L}_0)=0$. By using Serre duality, we can further reduce to the case $h^2(\mathcal{L}_0)=0$. Then, by semi-continuity, we have $h^2(\mathcal{L}_t)=0$ for all sufficiently small $t$. Thus, by constancy of the holomorphic Euler characteristic~$\chi(\mathcal{L}_t)$, it remains to show that $h^0(\mathcal{L}_t)$ is constant, i.e., that any section $s\in H^0(\mathcal{L}_0)$ can be extended to give a section $\tilde{s}\in H^0(\mathcal{L}_U)$ with $\tilde{s}|_{\mathcal{S}_0}=s$.

So, let $s\in H^0(\mathcal{L}_0)$. Then $s$ defines an effective divisor $D=\sum_i D_i$, where the $D_i$ are (not necessarily distinct) prime divisors. By hypothesis, there exist line bundles $\mathcal{L}_i$ with $\mathcal{L}_i|_{\mathcal{S}_0}\cong\mathcal{O}_{\mathcal{S}_0}(D_i)$. Now, since $\mathcal{S}_0$ is a non-algebraic K3 surface, every $D_i$ is either a $(-2)$-curve or the reduction of a fiber of an elliptic fibration of~$S$. In either case, an easy calculation shows that $h^1(\mathcal{O}_{\mathcal{S}_0}(D_i))=h^2(\mathcal{O}_{\mathcal{S}_0}(D_i))=0$. By semi-continuity and constancy of holomorphic Euler characteristic, this implies that $h^0(\mathcal{L}_i|_{\mathcal{S}_t})$ is independent of~$t$. So, if we write $s=\prod_is_i$ with $s_i\in H^0(\mathcal{O}_{\mathcal{S}_0}(D_i))$, the $s_i$ can be extended to give sections $\tilde{s}_i\in H^0(\mathcal{L}_i)$ with $\tilde{s}_i|_{\mathcal{S}_0}=s_i$. But now $(\bigotimes_i\mathcal{L}_i)\otimes \mathcal{L}^*$ is numerically trivial on each $\mathcal{S}_t$ and thus trivial, so $\bigotimes_i\mathcal{L}_i\cong\mathcal{L}$ and the product $\tilde{s}:=\prod_i\tilde{s}_i$ defines a section in $H^0(\mathcal{L})$ with $\tilde{s}|_{\mathcal{S}_0}=s$.
\end{proof}
For tori, we need to choose the extended bundles appropriately:
\begin{prop}\label{prop:torcoh}
Let $S$ be a non-algebraic torus and $\pi\colon\mathcal{S}\to T$ the deformation of $S=\mathcal{S}_0$ from Proposition~\ref{prop:k3torlb}. Then for any line bundle $L$ on $S$, there exists a line bundle $\mathcal{L}$ on $\mathcal{S}$ and an open neighborhood $U\subset T$ of $0$ such that $\mathcal{L}|_{\mathcal{S}_0}\cong L$ and $U\ni t\mapsto h^q(\mathcal{L}|_{\mathcal{S}_t})$ is a constant map for every~$q$.
\end{prop}
\begin{proof}
As in the proof of Proposition~\ref{prop:k3coh}, we can assume $h^2(L)=0$ and $h^0(L)\ne 0$. We then have $L\cong\mathcal{O}_S(D)$ for some divisor $D>0$ on~$S$. Since there are no curves on a torus of algebraic dimension~$0$, we have $a(S)=1$ and thus obtain an elliptic fibration $r\colon S\to C$ over some elliptic curve~$C$. Now all irreducible curves on $S$ are fibers of $r$, so there exists an effective divisor $\delta$ on~$C$ such that $D=r^*\delta$.

Since $L$ extends to $\mathcal{S}$ by Proposition~\ref{prop:k3torlb}, we have $a(\mathcal{S}_t)\ge1$ for all $t\in T$, thus, there exists a relative algebraic reduction $R\colon\mathcal{S}\to\mathcal{C}$ over $T$, where $\mathcal{C}\to T$ is a deformation of $C=\mathcal{C}_0$ and $R|_{\mathcal{S}_0}=r$. We can now choose an effective divisor $\tilde{\delta}$ on $\mathcal{C}$, which is flat over $T$, such that $\tilde{\delta}|_{\mathcal{S}_0}=\delta$. We then obtain the desired extension of the line bundle~$L$ by letting $\mathcal{L}:=R^*\mathcal{O}_{\mathcal{C}}(\tilde{\delta})$.
\end{proof}
The preceding propositions enable us to extend rank-two bundles having a section:
\begin{lem}\label{lem:k3torr2}
Let $S$ be a K3 surface or a torus and $V$ a rank-two vector bundle on $S$ with $H^0(V)\ne0$. Then there exists an algebraic approximation $\pi\colon\mathcal{S}\to T\ni0$ of $S=\mathcal{S}_0$ and a rank-two bundle $\mathcal{V}$ on $\mathcal{S}$ such that $\mathcal{V}|_{\mathcal{S}_0}\cong V$.
\end{lem}
\begin{proof}
Let $0\ne s\in H^0(V)$. If we denote by $G$ the divisorial part of the zero locus of~$s$, we obtain a short exact sequence
\begin{equation}\label{eq:r2seq}
0\longrightarrow\mathcal{O}_S\longrightarrow V\otimes\mathcal{O}_S(-G)\longrightarrow L\otimes\mathcal{I}_Y\longrightarrow0,
\end{equation}
where $L$ is some line bundle on~$S$, and $\mathcal{I}_Y$ is the ideal sheaf of some locally complete intersection subvariety $Y\subset S$ of codimension at least~$2$. We observe that sequence~\eqref{eq:r2seq} is given by an extension class in $\Ext^1(L\otimes\mathcal{I}_Y,\mathcal{O}_S)$, which by Serre duality is isomorphic to $H^1(L\otimes\mathcal{I}_Y)$.

Suppose we have managed to construct an algebraic approximation $\pi\colon\mathcal{S}\to T\ni0$ of $S=\mathcal{S}_0$, a line bundle $\mathcal{L}$ on $\mathcal{S}$ with $\mathcal{L}|_{\mathcal{S}_0}\cong L$ and a locally complete intersection subvariety $\mathcal{Y}\subset\mathcal{S}$, flat over $T$, with $\mathcal{Y}\cap\mathcal{S}_0=Y$. We define $\mathcal{Y}_t:=\mathcal{Y}\cap\mathcal{S}_t$ and suppose further that $h^1(\mathcal{L}_t\otimes\mathcal{I}_{\mathcal{Y}_t})$ is independent of~$t$ (we remark that flatness of $\mathcal{Y}$ implies $\mathcal{I}_{\mathcal{Y}}|_{\mathcal{S}_t}\cong\mathcal{I}_{\mathcal{Y}_t}$). Then the extension class in $H^1(L\otimes\mathcal{I}_Y)=H^1(\mathcal{L}_0\otimes\mathcal{I}_{\mathcal{Y}_0})$ can be extended to give a class in $H^1(\mathcal{L}\otimes\mathcal{I}_{\mathcal{Y}})$. We thus obtain a short exact sequence
\[0\longrightarrow\mathcal{O}_{\mathcal{S}}\longrightarrow\mathcal{V}'\longrightarrow\mathcal{L}\otimes\mathcal{I}_{\mathcal{Y}}\longrightarrow0\]
defining a rank-$2$ vector bundle $\mathcal{V}'$ on~$\mathcal{S}$ with $\mathcal{V}'|_{\mathcal{S}_0}\cong V\otimes\mathcal{O}_S(-G)$. If we have chosen the algebraic approximation such that every effective divisor can be extended as a line bundle, there is a line bundle $\mathcal{G}$ on $\mathcal{S}$ with $\mathcal{G}|_{\mathcal{S}_0}\cong\mathcal{O}_S(D)$, and we can just set $\mathcal{V}:=\mathcal{V}'\otimes\mathcal{G}$.

We now describe how to satisfy the assumptions we made above: We first remark that for any deformation $\pi\colon\mathcal{S}\to T$ of $S=\mathcal{S}_0$, any line bundle $\mathcal{L}$ on $\mathcal{S}$ and any codimension $\ge2$ subvariety $\mathcal{Y}\subset\mathcal{S}$ flat over $T$, we have for every $t\in T$ a short exact sequence
\[
0\longrightarrow(\mathcal{L}\otimes\mathcal{I}_{\mathcal{Y}})|_{\mathcal{S}_t}\longrightarrow\mathcal{L}_t\longrightarrow\mathcal{L}|_{\mathcal{Y}_t}\longrightarrow0,
\]
whose long exact cohomology sequence shows that $h^1(\mathcal{L}_t\otimes\mathcal{I}_{\mathcal{Y}_t})$ is constant provided that $h^0(\mathcal{L}_t\otimes\mathcal{I}_{\mathcal{Y}_t})$ and all the $h^q(\mathcal{L}_t)$ are constant.

Now, if $h^0(L\otimes\mathcal{I}_Y)=0$, we can just choose $\pi$ according to Proposition~\ref{prop:k3torlb}. We can then choose $\mathcal{L}$ with the desired properties according to Propositions \ref{prop:k3coh} and \ref{prop:torcoh}. The constancy of $h^0(\mathcal{L}_t\otimes\mathcal{I}_{\mathcal{Y}_t})$ then follows by semi-continuity for any flat extension $\mathcal{Y}\subset\mathcal{S}$ of~$Y$.

So, we can now assume that $Y\ne\emptyset$ and that $h^0(L\otimes\mathcal{I}_Y)>0$. Then also $h^0(L)>0$ and we have $L\cong\mathcal{O}_S(D)$ for some divisor $D>0$ on $S$. We first consider the case $a(S)=0$: Then $S$ is a K3 surface and $D$ is composed of $(-2)$-curves intersecting (at most) transversally. Furthermore, we have $h^0(L)=1$ and thus $Y\subset D$ as an inclusion of complex spaces. If we choose the algebraic approximation $\pi\colon\mathcal{S}\to T$ according to Proposition~\ref{prop:k3torlb}, we can find by Propositions \ref{prop:k3torlb} and \ref{prop:k3coh} a unique effective divisor $\mathcal{D}$ on $\mathcal{S}$ with $\mathcal{D}\cap\mathcal{S}_0=D$. Letting $\mathcal{D}_t:=\mathcal{D}\cap\mathcal{S}_t$ for $t\in T$, each $\mathcal{D}_t$ is again composed of $(-2)$-curves sharing the same intersection pattern as $D$. Since the $Y$ consists of isolated points only (together with infinitesimal neighborhoods) we can obviously construct a codimension $\ge2$ subvariety $\mathcal{Y}\subset\mathcal{S}$, flat over $T$ with $\mathcal{Y}_0=Y$, such that $\mathcal{Y}\subset\mathcal{D}$ as an inclusion of complex spaces. Setting $\mathcal{L}:=\mathcal{O}_{\mathcal{S}}(\mathcal{D})$, this means that the defining section of $\mathcal{D}_t$ inside $H^0(\mathcal{L}_t)$ actually yields a section inside $H^0(\mathcal{L}_t\otimes\mathcal{I}_{\mathcal{Y}_t})$. We have thus proved $h^0(\mathcal{L}_t\otimes\mathcal{I}_{\mathcal{Y}_t})=1$ for any~$t$.

It remains to consider the case that $a(S)=1$: We first assume $S$ to be a K3 suface. Then we choose the algebraic approximation $\pi\colon\mathcal{S}\to T$ according to Proposition~\ref{prop:elsurf}. As before, we can deform any divisor on $\mathcal{S}$ to a divisor on the neighboring $\mathcal{S}_t$ such that locally, its structure remains unchanged. Thus, by suitably choosing $\mathcal{Y}\subset\mathcal{S}$, we can ensure extendability of all sections in $H^0(\mathcal{S}_t\otimes\mathcal{I}_{\mathcal{Y}_t})$ as before.

If now $S$ is a torus with $a(S)=1$, we investigate the linear system of divisors on $S$ defined by the subspace $H^0(L\otimes\mathcal{I}_Y)\subset H^0(L)$. We can decompose this linear system as
\begin{equation}\label{eq:divdec}
\lvert H^0(L\otimes\mathcal{I}_Y)\rvert=D_0+\lvert L\otimes\mathcal{O}_S(-D_0)\rvert,
\end{equation}
where $D_0$ is an effective divisor on $S$ containing $Y$ as a complex subspace. We choose the algebraic approximation $\pi\colon\mathcal{S}\to T$ according to Proposition~\ref{prop:k3torlb} and line bundles $\mathcal{M}$ and $\mathcal{N}$ on $\mathcal{S}$ with $\mathcal{M}|_{\mathcal{S}_0}\cong\mathcal{O}_S(D_0)$ and $\mathcal{N}|_{\mathcal{S}_0}\cong L\otimes\mathcal{O}_S(-D_0)$ according to Proposition~\ref{prop:torcoh}. We obtain a section $\tilde{s}_0\in H^0(\mathcal{M})$ such that $\tilde{s}_0|_{\mathcal{S}_0}$ is just the defining section of the divisor $D_0$. By choosing $\tilde{s}_0$ suitably, we can achieve that there exists a codimension $\ge2$ subvariety $\mathcal{Y}\subset\mathcal{S}$ which is flat over $T$, such that $\mathcal{Y}_0=Y$ and $\tilde{s}_0$ vanishes along $\mathcal{Y}$. If we now let $\mathcal{L}:=\mathcal{M}\otimes\mathcal{N}$, we can extend any section $\zeta\in H^0(L\otimes\mathcal{I}_Y)$ to a section in $H^0(\mathcal{L}\otimes\mathcal{I}_{\mathcal{Y}})$ by writing $\zeta=s_0\cdot \zeta_0$ according to~\eqref{eq:divdec}, extending $\zeta_0$ to a section $\tilde{\zeta}_0\in H^0(\mathcal{N})$ according to Proposition~\ref{prop:torcoh} and choosing $\tilde{s}_0\cdot\tilde{\zeta}_0$ as extension of $\zeta$.

To conclude the proof, we remark that each of the algebraic approximations of $S$ chosen above has the property that any effective divisor on $S$ extends as a line bundle on $\mathcal{S}$ (cf.~Propositions \ref{prop:elsurf} and \ref{prop:k3torlb}), so the line bundle $\mathcal{G}$ mentioned at the beginning indeed exists.
\end{proof}
Using different arguments, we can also treat the case of simple rank-two bundles:
\begin{lem}\label{lem:simpbun}
Let $S$ be a K3 surface or a torus. Then there exists an algebraic approximation $\pi\colon\mathcal{S}\to T\ni 0$ of $S=\mathcal{S}_0$ such that for any simple rank-two vector bundle $V$ on $S$, there exists a rank-two bundle $\mathcal{V}$ on $\mathcal{S}$ such that $\mathcal{V}|_{\mathcal{S}_0}\cong V$.
\end{lem}
\begin{proof}
We let $\pi\colon\mathcal{S}\to T$ be the algebraic approximation of $S=\mathcal{S}_0$ according to Proposition~\ref{prop:k3torlb}. We consider the projective bundle $\mathbb{P}(V)$ over $S$. An easy calculation shows that $h^2(T_{\mathbb{P}(V)/S})=h^2(S^2V\otimes\det V^*)=0$. This implies by~\cite[Theorem~6.1]{Hor74} that there exists a deformation $\psi\colon\mathcal{X}\to T$ of $\mathbb{P}(V)=\mathcal{X}_0$ and a holomorphic map $F\colon\mathcal{X}\to\mathcal{S}$ extending the natural projection $f\colon\mathbb{P}(V)\to S$. We observe that the bundle $V$ is given by
$V=f_*\mathcal{O}_{\mathbb{P}(V)}(1)$ and that $\mathcal{O}_{\mathbb{P}(V)}(2)\cong K_{\mathbb{P}(V)/S}^*\otimes f^*\det V$. By Proposition~\ref{prop:k3torlb}, there exists a line bundle $\mathcal{G}$ on $\mathcal{S}$ with $\mathcal{G}|_{\mathcal{S}_0}\cong\det V$. We define a line bundle $\mathcal{L}$ on $\mathcal{X}$ as $\mathcal{L}:=K_{\mathcal{X}/\mathcal{S}}\otimes F^*\mathcal{G}$. Now obviously the bundle $\mathcal{L}_0\cong\mathcal{O}_{\mathbb{P}(V)}(2)$ has a square root. The obstruction to having a square root is purely topological, so by the topological triviality of~$\psi$, also $\mathcal{L}$ has a square root, i.e.~a line bundle $\mathcal{K}$ with $\mathcal{K}^{\otimes2}\cong\mathcal{L}$. We can then define $\mathcal{V}:=F_*\mathcal{L}$ to obtain a rank-$2$ bundle on $\mathcal{S}$ with $\mathcal{V}|_{\mathcal{S}_0}\cong V$.
\end{proof}
The following lemma allows to carry over the above statements to blown-up K3 surfaces and tori:
\begin{lem}\label{lem:blowind}
Let $S$ be a compact surface, $p\colon\hat{S}\to S$ the blow-up of a point $s_0\in S$ and $\hat{V}$ a vector bundle on $\hat{S}$. Suppose that there exists an algebraic approximation $\pi\colon\mathcal{S}\to T\ni0$ of $S=\mathcal{S}_0$ and a vector bundle $\mathcal{V}$ on $\mathcal{S}$ such that
\[\mathcal{V}|_{\mathcal{S}_0}\cong (p_*\hat{V})^{**}.\]
Then, after shrinking $T$ sufficiently, there exists an algebraic approximation $\hat{\mathcal{S}}\to T$ of $\hat{S}=\hat{\mathcal{S}}_0$, a proper modification $P\colon\hat{\mathcal{S}}\to\mathcal{S}$ over $T$ with $P|_{\hat{\mathcal{S}}_0}=p$ and a vector  bundle $\hat{\mathcal{V}}$ on $\hat{\mathcal{S}}$ such that
\[\hat{\mathcal{V}}|_{\hat{\mathcal{S}}_0}\cong\hat{V}.\]
\end{lem}
\begin{proof}
We choose a submanifold $\mathcal{C}\subset\mathcal{S}$ which is mapped isomorphically onto $T$ by $\pi$ such that $\mathcal{C}\cap\mathcal{S}_0=\{s_0\}$. We now let $P\colon\hat{\mathcal{S}}\to\mathcal{S}$ be the blow-up of $\mathcal{C}\subset\mathcal{S}$. The deformation $\hat{\pi}:=\pi\circ P\colon\hat{\mathcal{S}}\to T$ is then an algebraic approximation of $\hat{S}=\hat{\mathcal{S}}_0$.

It remains to construct the desired bundle $\hat{\mathcal{V}}$ on $\hat{\mathcal{S}}$: We let $V:=(p_*\hat{V})^{**}$. Since $V$ is a reflexive sheaf on the two-dimensional manifold $S$, it is locally free. The bidual of the natural map $p^*p_*\hat{V}\to\hat{V}$ induces a map
\[\iota\colon p^*V\to\hat{V}^{**}=\hat{V}.\]
If we let $\mathcal{E}\subset\hat{\mathcal{S}}$ be the exceptional divisor of $P$, then $E:=\mathcal{E}\cap\mathcal{S}_0\subset\mathcal{S}_0$ is the exceptional divisor of $p$ and $\iota$ is an isomorphism over $\hat{\mathcal{S}}_0\setminus E$. We thus obtain a short exact sequence
\begin{equation}\label{eq:blowindseq}
0\longrightarrow p^*V\longrightarrow\hat{V}\longrightarrow Q\longrightarrow 0,
\end{equation}
given by a cohomology class in~$\Ext^1(Q,p^*V)$, where $Q$ is a coherent sheaf on $\hat{\mathcal{S}}_0$ with support contained in $E$. We want to construct an extension of sequence~\eqref{eq:blowindseq} to all of $\hat{\mathcal{S}}$.

By hypothesis, there is a vector bundle $\mathcal{V}$ on $\mathcal{S}$ with $\mathcal{V}|_{\mathcal{S}_0}\cong V$. We can choose an open neighborhood of $\mathcal{E}$ inside $\hat{\mathcal{S}}$ such that $P^*\mathcal{V}|_{\mathcal{U}}$ is trivial and such that $\hat{\pi}$ induces an isomorphism $\mathcal{U}\cong U\times T$ for some open neighborhood $U$ of $E$ in $\hat{\mathcal{S}}_0$. Via this isomorphism, we can extend the sheaf $Q$ constantly to give a coherent sheaf on $\mathcal{U}$. The support of this sheaf being contained in $\mathcal{E}$, we can extend by zero to obtain a coherent sheaf $\mathcal{Q}$ on~$\hat{\mathcal{S}}$. By construction, we have for every point $t\in T$:
\[\Ext^1(\mathcal{Q}|_{\hat{\mathcal{S}}_t},(P^*\mathcal{V})|_{\hat{\mathcal{S}}_t})\cong\Ext^1(Q|_U,\mathcal{O}_U),\]
so in particular, the dimension
\[\dim\Ext^1(\mathcal{Q}|_{\hat{\mathcal{S}}_t},(P^*\mathcal{V})|_{\hat{\mathcal{S}}_t})\]
is independent of~$t$. By~\cite[Satz~3]{BPS80}, this implies that the base change morphism
\[\EExt^1_{\hat{\pi}}(\mathcal{Q},P^*\mathcal{V})\otimes\mathbb{C}(t)\to\Ext^1(\mathcal{Q}|_{\hat{\mathcal{S}}_t},(P^*\mathcal{V})|_{\hat{\mathcal{S}}_t})\]
is an isomorphism for any $t\in T$. This means that (after shrinking $T$ sufficiently) we can extend the class of sequence~\eqref{eq:blowindseq} to a class in $\Ext^1(\mathcal{Q},P^*\mathcal{V})$ giving a short exact sequence
\[0\longrightarrow P^*\mathcal{V}\longrightarrow\hat{\mathcal{V}}\longrightarrow\mathcal{Q}\longrightarrow0,\]
which defines a bundle $\hat{\mathcal{V}}$ as desired.
\end{proof}
We can now prove the central result of this section:
\begin{proof}[Proof of Theorem~\ref{thm:projr2b}]
It is sufficient to construct an algebraic approximation $\pi\colon\mathcal{S}\to T$ of $S=\mathcal{S}_0$ such that there exists a vector bundle $\mathcal{V}$ on $\mathcal{S}$ with $\mathcal{V}|_{\mathcal{S}_0}\cong V\otimes L$ for some line bundle $L$ on $S$ (we can then take $\mathbb{P}(\mathcal{V})$ as an algebraic approximation of $\mathbb{P}(V)$). By Lemma~\ref{lem:blowind}, we can assume $S$ to be minimal, so $S$ is either a K3 surface or a torus. Now, if $H^0(V \otimes L)\ne0$ for some $L$, we are done by Lemma~\ref{lem:k3torr2}. If $H^0(V\otimes L)=0$ for every line bundle $L$ on $S$, then in particular $V$ is simple, so we can apply Lemma~\ref{lem:simpbun}.
\end{proof}

\section{Further results}
\label{sec:algebr-appr-cert}
\subsection{Threefolds bimeromorphic to a product}
\begin{lem}\label{lem:surfapp}
Let $S$ be a compact Kähler surface. Then there exists an algebraic approximation $\pi\colon\mathcal{S}\to T\ni0$ of $S$ with the following property: For any compact curve $C\subset S$ there exists an open subset $U\subset S$ with $C\subset U$ and an open subset $\mathcal{U}\subset\mathcal{S}$ such that
\begin{enumerate}[(i)]
\item $\mathcal{U}\cap\mathcal{S}_0=U$ and
\item $\mathcal{U}$ is isomorphic to the product $U\times T$ over $T$.
\end{enumerate}
\end{lem}
\begin{proof}
If $a(S)=1$ then $S$ is elliptic and any connected compact curve on $S$ is contained in a fiber a the elliptic fibration. The claim then follows directly from Proposition~\ref{prop:elsurf}.

If $a(S)=0$, we first observe that we can assume $S$ to be minimal. There are no compact curves on a torus of algebraic dimension $0$, so we are left with the case that $S$ is a K3 surface. Then $C$ is a configuration of $(-2)$-curves. We take $\pi$ to be the deformation from Proposition~\ref{prop:k3torlb}. Then, by Proposition~\ref{prop:k3coh}, $C$ can be extended to a curve $\mathcal{C}\subset\mathcal{S}$ which is proper and flat over~$T$, such that for each $t\in T$, the intersection $\mathcal{C}_t:=\mathcal{C}\cap\mathcal{S}_t$ is a configuration of $(-2)$-curves isomorphic to~$C$. The claim now follows because isomorphic configurations of $(-2)$-curves have isomorphic neighborhoods.
\end{proof}
Lemma~\ref{lem:surfapp} is the essential ingredient to the
\begin{proof}[Proof of Theorem~\ref{thm:main}] By Hironaka's Chow Lemma (cf.~\cite{Hir75}) and resolution of singularities, there exists a compact complex manifold $\hat X$ together with a holomorphic map $p\colon \hat{X}\to X$ obtained as a finite composition of blow-ups along smooth centers such that there exists a proper modification $f\colon\hat{X}\to \mathbb{P}_1\times S$ making the diagram
\[\xymatrix{
&\hat{X}\ar[d]^p\ar[dl]_f\\
\mathbb{P}_1\times S\ar@{-->}[r]_\phi&X
}\]
commute. The center of $f$ has codimension at least $2$ in $\mathbb{P}_1\times S$, so there exists a (not necessarily connected) compact curve $C\subset S$ such that $f$ is an isomorphism over $\mathbb{P}_1\times(S\setminus C)$.

We now take an algebraic approximation $\pi\colon\mathcal{S}\to T$ of $S$ according to Lemma~\ref{lem:surfapp}. Then there is an open neighborhood $U\subset S$ of $C$ and an open subset $\mathcal{U}\subset\mathcal{S}$ with $\mathcal{U}\cap\mathcal{S}_0=U$ and $\mathcal{U}\cong U\times T$ over $T$. We observe that $\tilde{\pi}:=\pr_2\circ(\id\times\pi)\colon\mathbb{P}_1\times\mathcal{S}\to T$ is an algebraic approximation of~$\mathbb{P}_1\times S$ such that the restriction $\tilde{\pi}|_{\mathbb{P}_1\times\mathcal{U}}$ is isomorphic to the natural projection $\mathbb{P}_1\times U\times T\to T$. The modification $f$ thus trivially induces a modification $\tilde{f}\colon f^{-1}(\mathbb{P}_1\times U)\times T\to\mathbb{P}_1\times\mathcal{U}$ which can be glued together with the identity on $\mathbb{P}_1\times(\mathcal{S}\setminus\mathcal{U})$ to give a proper modification $F\colon\hat{\mathcal{X}}\to\mathbb{P}_1\times\mathcal{S}$ such that the composition $\tilde{\pi}\circ F\colon\hat{\mathcal{X}}\to T$ is an algebraic approximation of~$\hat{X}$.

Since $p$ is a composition of blow-ups along smooth centers, we have $p_*\mathcal{O}_{\hat{X}}=\mathcal{O}_X$ and $R^1p_*\mathcal{O}_{\hat{X}}=0$. Using \cite[Theorem~8.2]{Hor76} we conclude (after shrinking $T$ sufficiently) the existence of a proper holomorphic submersion $\psi\colon\mathcal{X}\to T$ and a holomorphic map $P\colon\hat{\mathcal{X}}\to\mathcal{X}$ such that $P\circ\psi=\tilde{\pi}\circ F$, $\psi^{-1}(0)\cong X$ and $P|_{(\tilde{\pi}\circ F)^{-1}(0)}=p$. It is now easy to see that $\psi$ is an algebraic approximation of $X$.
\end{proof}
\begin{rem}
By Fujiki's classification results (\cite[Proposition~14.1]{Fuj83}), every non-algebraic uniruled Kähler threefold carries a meromorphic $\mathbb{P}_1$-fibration.
\end{rem}

\subsection{Conic bundles}
\begin{proof}[Proof of Theorem~\ref{thm:ellcon}]
We let $E:=f_*K_{X/S}^*$. By hypothesis, $r_*E$ is a rank-3 bundle on $C$ and we get an exact sequence
\begin{equation}\label{eq:ellseq}
0\longrightarrow r^*r_*E\stackrel{\alpha}{\longrightarrow}E\longrightarrow Q\longrightarrow 0,
\end{equation}
where $Q$ is a torsion sheaf on~$S$. There exists a finite set $Z\subset C$ such that the support of $Q$ is contained in $r^{-1}(Z)$. This means that $\alpha|_{r^{-1}(C\setminus Z)}$ is an isomorphism.

We choose an algebraic approximation $\pi\colon\mathcal{S}\to T\ni0$ of $S=\mathcal{S}_0$ and an extension $R\colon\mathcal{S}\to C\times T$ of~$r$ according to Proposition~\ref{prop:elsurf}. By the local triviality of $R$ over $T$, there exists a coherent sheaf $\mathcal{Q}$ on $\mathcal{S}$, flat over $T$, with $\mathrm{supp}(\mathcal{Q})\subset R^{-1}(Z\times T)$, such that $\mathcal{Q}|_{\mathcal{S}_0}\cong Q$ and
\[\Ext^1(\mathcal{Q}|_{\mathcal{S}_t},\mathcal{O}_{\mathcal{S}_t})\cong\Ext^1(Q,\mathcal{O}_S)\]
for all $t\in T$. This implies the existence of a rank-$3$ bundle $\mathcal{E}$ on $\mathcal{S}$ sitting in a short exact sequence
\begin{equation}\label{eq:ellseqext}
0\longrightarrow R^*\pr_1^*r_*E\stackrel{\tilde{\alpha}}{\longrightarrow}\mathcal{E}\longrightarrow\mathcal{Q}\longrightarrow0
\end{equation}
whose restriction to~$\mathcal{S}_0$ yields sequence~\eqref{eq:ellseq}. In particular, $\tilde{\alpha}|_{R^{-1}((C\setminus Z)\times T)}$ is an isomorphism. By construction, the section
\[\sigma\in H^0(S^2E\otimes\det E^*)\]
yields a section
\[\tilde{\sigma}\in H^0\left((S^2\mathcal{E}\otimes\det\mathcal{E}^*)|_{R^{-1}((C\setminus Z)\times T)}\right)\]
with $\tilde{\sigma}|_{R^{-1}((C\setminus Z)\times\{0\})}=\sigma|_{r^{-1}(C\setminus Z)}$. By the local triviality of the situation, however, any point $p\in Z$ posesses an open neighborhood $U\subset C$ such that there exists a section
\[\tilde{\tilde{\sigma}}\in H^0((S^2\mathcal{E}\otimes\det\mathcal{E}^*)|_{R^{-1}(U\times T)})\]
with $\tilde{\tilde{\sigma}}|_{R^{-1}(U\times\{0\})}=\sigma|_{r^{-1}(U)}$ and $\tilde{\tilde{\sigma}}|_{R^{-1}((U\setminus\{p\})\times T)}=\tilde{\sigma}|_{R^{-1}((U\setminus\{p\})\times T)}$. The sections $\tilde{\sigma}$ and $\tilde{\tilde{\sigma}}$ glue together to give a section
\[s\in H^0(S^2\mathcal{E}\otimes\det\mathcal{E}^*)\]
with $s|_{\mathcal{S}_0}=\sigma$. The section~$s$ defines a conic bundle over~$\mathcal{S}$ yielding an algebraic approximation of~$X$.
\end{proof}

\begin{proof}[Proof of Theorem~\ref{thm:splitcon}]
Let $E:=f_*K_{X/S}^*$. Then $X\subset\mathbb{P}(E)$ is given by a section $s\in H^0(S^2E\otimes\det E^*)$.

Let $\pi\colon\mathcal{S}\to T$ be the algebraic approximation of $S=\mathcal{S}_0$ from Proposition~\ref{prop:k3torlb}. We write $E\cong L_1\oplus L_2\oplus L_3$ and choose line bundles $\mathcal{L}_i$ on $\mathcal{S}$ with $\mathcal{L}_i|_{\mathcal{S}_0}\cong L_i$ for $i=1$, $2$, $3$. Letting $\mathcal{E}:=\mathcal{L}_1\oplus\mathcal{L}_2\oplus\mathcal{L}_3$, we have $\mathcal{E}|_{\mathcal{S}_0}\cong E$, and by Proposition~\ref{prop:k3coh}, the section $s$ extends to a section $\tilde{s}\in H^0(S^2\mathcal{E}\otimes\det\mathcal{E}^*)$ with $\tilde{s}|_{\mathcal{S}_0}= s$, defining a conic bundle $\mathcal{X}\subset\mathbb{P}(\mathcal{E})$ over $\mathcal{S}$ yielding an algebraic approximation of $\mathcal{X}_0\cong X$.
\end{proof}

\bibliography{literatur}

\providecommand{\bysame}{\leavevmode\hbox to3em{\hrulefill}\thinspace}
\providecommand{\MR}{\relax\ifhmode\unskip\space\fi MR }
\providecommand{\MRhref}[2]{%
  \href{http://www.ams.org/mathscinet-getitem?mr=#1}{#2}
}
\providecommand{\href}[2]{#2}
\begin{thebibliography}{BHPVdV04}

\bibitem[BHPVdV04]{BPV04}
Wolf~P. Barth, Klaus Hulek, Chris A.~M. Peters, and Antonius Van~de Ven,
  \emph{Compact complex surfaces}, 2 ed., Ergebnisse der Mathematik und ihrer
  Grenzgebiete. 3. Folge., no.~4, Springer-Verlag, Berlin, 2004.

\bibitem[BLP87]{BL87}
C.~B{\u{a}}nic{\u{a}} and J.~Le~Potier, \emph{Sur l'existence des fibr\'es
  vectoriels holomorphes sur les surfaces non-alg\'ebriques}, J. Reine Angew.
  Math. \textbf{378} (1987), 1--31.

\bibitem[BPS80]{BPS80}
C.~B{\u{a}}nic{\u{a}}, M.~Putinar, and G.~Schumacher, \emph{Variation der
  globalen {E}xt in {D}eformationen kompakter komplexer {R}\"aume}, Math. Ann.
  \textbf{250} (1980), no.~2, 135--155.

\bibitem[Buc06]{Buc06}
Nicholas Buchdahl, \emph{Algebraic deformations of compact {K}\"ahler
  surfaces}, Math. Z. \textbf{253} (2006), no.~3, 453--459.

\bibitem[Buc08]{Buc08}
\bysame, \emph{Algebraic deformations of compact {K}\"ahler surfaces. {II}},
  Math. Z. \textbf{258} (2008), no.~3, 493--498.

\bibitem[FM94]{FM94}
Robert Friedman and John~W. Morgan, \emph{Smooth four-manifolds and complex
  surfaces}, Ergebnisse der Mathematik und ihrer Grenzgebiete. 3. Folge.,
  no.~27, Springer-Verlag, Berlin, 1994.

\bibitem[Fuj83a]{Fuj83}
Akira Fujiki, \emph{On the structure of compact complex manifolds in {${\cal
  C}$}}, Algebraic varieties and analytic varieties ({T}okyo, 1981), Adv. Stud.
  Pure Math., no.~1, North-Holland, Amsterdam, 1983, pp.~231--302.

\bibitem[Fuj83b]{Fuj83a}
\bysame, \emph{Relative algebraic reduction and relative {A}lbanese map for a
  fiber space in {${\cal C}$}}, Publ. Res. Inst. Math. Sci. \textbf{19} (1983),
  no.~1, 207--236.

\bibitem[Hir75]{Hir75}
Heisuke Hironaka, \emph{Flattening theorem in complex-analytic geometry}, Amer.
  J. Math. \textbf{97} (1975), 503--547.

\bibitem[Hor74]{Hor74}
Eiji Horikawa, \emph{On deformations of holomorphic maps. {II}}, J. Math. Soc.
  Japan \textbf{26} (1974), 647--667.

\bibitem[Hor76]{Hor76}
\bysame, \emph{On deformations of holomorphic maps. {III}}, Math. Ann.
  \textbf{222} (1976), no.~3, 275--282.

\bibitem[Kod63]{Kod63b}
K.~Kodaira, \emph{On compact analytic surfaces. {III}}, Ann. of Math. (2)
  \textbf{78} (1963), 1--40.

\bibitem[Miy83]{Miy83}
Masayoshi Miyanishi, \emph{Algebraic methods in the theory of algebraic
  threefolds---surrounding the works of {I}skovskikh, {M}ori and {S}arkisov},
  Algebraic varieties and analytic varieties ({T}okyo, 1981), Adv. Stud. Pure
  Math., no.~1, North-Holland, Amsterdam, 1983, pp.~69--99.

\bibitem[Ran91]{Ran91}
Ziv Ran, \emph{Stability of certain holomorphic maps}, J. Differential Geom.
  \textbf{34} (1991), no.~1, 37--47.

\bibitem[Sch10]{Schrack10}
Florian Schrack, \emph{Algebraische {A}pproximation von
  {K}\"ahlermannigfaltigkeiten}, Ph.D. thesis, Universit\"at Bayreuth, December
  2010, \url{http://opus.ub.uni-bayreuth.de/volltexte/2010/754/}.

\bibitem[Voi04]{Voi04}
Claire Voisin, \emph{On the homotopy types of compact {K}\"ahler and complex
  projective manifolds}, Invent. Math. \textbf{157} (2004), no.~2, 329--343.

\bibitem[Yan89]{Yan89}
Jae-Hyun Yang, \emph{Holomorphic vector bundles over complex tori}, J. Korean
  Math. Soc. \textbf{26} (1989), no.~1, 117--142.

\end{thebibliography}
\end{document}